\documentclass[10pt]{article}
\usepackage[utf8]{inputenc}

\usepackage{palatino}
\usepackage{fullpage}

\usepackage[linesnumbered,ruled]{algorithm2e}

\usepackage{thm-restate}

\usepackage{hyperref}
\usepackage[dvipsnames]{xcolor}
\usepackage{todonotes}
\usepackage{xspace}
\usepackage{complexity}
\usepackage{microtype}
\usepackage[linesnumbered,ruled]{algorithm2e}
\usepackage{tikz}
\usetikzlibrary{patterns}
\usepackage{enumerate}
\usepackage{doi}
\usepackage{adjustbox}
\usepackage{graphicx}
\usepackage{authblk}
\usepackage{graphicx}
\usepackage{amsmath}
\usepackage{mathrsfs}
\usepackage{xcolor}

\usepackage{csquotes}
\usepackage[english]{babel}
\usepackage{eurosym}
\usepackage{palatino}
 \usepackage{verbatim}
 \usepackage{fancybox}
 \usepackage{comment}

\usepackage{subfig}
\usepackage{enumerate}
\usepackage{epsfig}
\usepackage[active]{srcltx}
\usepackage{amsfonts}
\usepackage{amsmath}
\usepackage[mathlines]{lineno}
\usepackage{xspace}
\usepackage{tikz}
\usetikzlibrary{arrows}
\usetikzlibrary{arrows,shapes}
\usetikzlibrary{arrows.meta}
\usepackage{authblk}
\usepackage{soul}

\usepackage{algorithmic}

\definecolor{red}{RGB}{255,0,0}
\definecolor{blue}{RGB}{0,0,255}

\usetikzlibrary{arrows,shapes,snakes,automata,backgrounds,petri,patterns}

\newcommand{\satrfull}{\textsc{Satisfiability Reconfiguration }}

\newcommand{\dsrfullts}{\textsc{Dominating Set Reconfiguration under Token Sliding }}

\newcommand{\satr}{\textsc{SATR }}
\newcommand{\satrnospace}{\textsc{SATR}}
\newcommand{\dsr}{\textsc{DSR$_{\mbox{TS}}$ }}
\newcommand{\dsrnospace}{\textsc{DSR$_{\mbox{TS}}$}}

\newcommand{\psc}{\textsc{PSPACE}-complete }
\newcommand{\pscnospace}{\textsc{PSPACE}-complete}
\newcommand{\npc}{\textsc{NP}-complete }

\definecolor{palette_bleuf}{HTML}{326C85}
\definecolor{palette_bleuc}{HTML}{679FB8}
\definecolor{palette_vert}{HTML}{399C54}
\definecolor{palette_jaune}{HTML}{FCD63E}
\definecolor{palette_beige}{HTML}{FAFEFF}
\definecolor{palette_rouge}{HTML}{9E0A1A}
\definecolor{palette_rose}{HTML}{FF9999}

\newenvironment{proofclaim}[1][]%
    {\noindent \emph{Proof.} {}{#1}{}}{\hfill
    $\Diamond$\vspace{1em}}
    
\usepackage{amsthm}
\newtheorem{theorem}{Theorem}

\newtheorem{lemma}{Lemma}

\newtheorem{corollary}{Corollary}

\newtheorem{claim}{Claim}

\newtheorem{remark}{Remark}

\usepackage{thmtools}

\providecommand{\keywords}[1]{\textbf{Keywords:} #1}

\title{TS-Reconfiguration of Dominating Sets in circle and circular-arc graphs \thanks{This work was supported by ANR project GrR (ANR-18-CE40-0032).}}

\author[1]{Nicolas Bousquet}
\author[1]{Alice Joffard}

\affil[1]{CNRS, LIRIS, Universit\'e de Lyon, Universit\'e Claude Bernard Lyon 1, Lyon, France
\thanks{firstname.lastname@liris.cnrs.fr}}

\date{}

\begin{document}

\maketitle

\begin{abstract}
We study the dominating set reconfiguration problem with  the token sliding rule. It consists, given a graph $G=(V,E)$ and two dominating sets $D_s$ and $D_t$ of $G$, in determining if there exists a sequence $S=<D_1:=D_s,\ldots,D_{\ell}:=D_t>$ of dominating sets of $G$ such that for any two consecutive dominating sets $D_r$ and $D_{r+1}$ with $r<t$, $D_{r+1}=D_r\setminus u \cup v$, where $uv\in E$. 

In a recent paper, Bonamy et al.~\cite{bonamy2019dominating} studied this problem and raised the following questions: what is the complexity of this problem on circular arc graphs? On circle graphs? In this paper, we answer both questions by proving that the problem is polynomial on circular-arc graphs and PSPACE-complete on circle graphs. \\

\noindent \keywords{reconfiguration, dominating sets, token sliding, circle graphs, circular arc graphs.}

\end{abstract}

\section{Introduction}\label{sec:Intro}

Reconfiguration problems consist, given an instance of a problem, in determining if (and in how many steps) we can transform one of its solutions into another one via a sequence of elementary operations keeping a solution along this sequence. The sequence is called a \emph{reconfiguration sequence}.

Let $\Pi$ be a problem and $\mathcal{I}$ be an instance of $\Pi$. Another way to describe a reconfiguration problem is to define the \emph{reconfiguration graph} $\mathcal{R}_\mathcal{I}$, whose vertices are the solutions of the instance $\mathcal{I}$ of $\Pi$, and in which two solutions are adjacent if and only if we can transform the first solution into the second in one elementary step. In this paper, we focus on the so-called \textsc{Reachability} problem which, given an instance $\mathcal{I}$ of a problem $\Pi$ and two solutions $I,J$ of $\mathcal{I}$, returns true if and only if there exists a reconfiguration sequence from $I$ to $J$ keeping a solution all along. Other works have focused on slightly different problems such as the connectivity of the reconfiguration graph or its diameter, see e.g.~\cite{bousquet2019polynomial,gopalan2009connectivity, haas2014k}.
Reconfiguration problems arise in various fields such as combinatorial games, motion of robots, random sampling, or enumeration.
Reconfiguration has been intensively studied for various rules and problems such as satisfiability constraints~\cite{gopalan2009connectivity}, graph coloring \cite{bonamy2019conjecture,CerecedaHJ11}, vertex covers and independent sets  \cite{HearnD05,ito2011complexity,lokshtanov2018complexity} or matchings~\cite{BonamyBHIKMMW19}. The reader is referred to the surveys~\cite{Nishimura17,van2013complexity} for a more complete overview on reconfiguration problems. In this work, we focus on dominating set reconfiguration. Throughout the paper, all the graphs are finite and simple.

Let $G=(V,E)$ be a graph. A \emph{dominating set} of $G$ is a subset of vertices $X$ such that, for every $v \in V$, either $v \in X$ or $v$ has a neighbor in $X$. A dominating set can be seen as a subset of tokens placed on vertices which dominates the graph. Three types of elementary operations, called \emph{reconfiguration rules}, have been studied for the reconfiguration of dominating sets.

\begin{itemize}
    \item The \emph{token addition-removal} rule (TAR) where each operation consists in either removing a token from a vertex, or adding a token on any vertex (keeping a dominating set).
    \item The \emph{token jumping} rule (TJ) where an operation consists in moving a token from a vertex to any vertex of the graph (keeping a dominating set).
    \item The \emph{token sliding} rule (TS) where an operation consists in sliding a token from a vertex to an adjacent vertex.
\end{itemize}

In this paper, we focus on the reconfiguration of dominating sets with the token sliding rule. Note that we authorize (as well as in the other papers on the topic, see~\cite{bonamy2019dominating}) the dominating sets to be multisets. In other words, several tokens can be put on the same vertex. Bonamy et al. observed in~\cite{bonamy2019dominating} that this choice can modify the reconfiguration graph and the set of dominating sets that can be reached from the initial one. More formally, we consider the following problem:

\smallskip
\noindent \dsrfullts (\dsr)\\
\textbf{Input:} A graph $G$, two dominating sets $D_s$ and $D_t$ of $G$.\\
\textbf{Output:} Does there exist a dominating set reconfiguration sequence from $D_s$ to $D_t$ under the token sliding rule ?

\paragraph{Dominating Set Reconfiguration under Token Sliding.}


The dominating set reconfiguration problem has been widely studied with the token addition-removal rule. Most of the earlier works focused on the conditions that ensure that the reconfiguration graph is connected in function of several graph parameters, see e.g.~\cite{bousquet2020linear,haas2014k,SuzukiMN14}. From a complexity point of view, Haddadan et al.~\cite{haddadan2016complexity}, proved that the reachability problem is \psc under the addition-removal rule, even when restricted to split graphs and bipartite graphs. They also provide linear time algorithms in trees and interval graphs.

More recently, Bonamy et al.~\cite{bonamy2019dominating} studied the dominating set reconfiguration problem under token sliding. They proved that \dsr is \pscnospace, even restricted to split, bipartite or bounded tree-width graphs. On the other hand, they provide polynomial time algorithms for cographs and dually chordal graphs (which contain interval graphs).
In their paper, they raise the following question: is it possible to generalize the polynomial time algorithm for interval graphs to circular arc-graphs ?

They also ask if there exists a class of graphs for which the maximum dominating set problem is NP-complete but its TS-reconfiguration counterpart is polynomial. They propose the class of circle graphs as a candidate. 


\paragraph{Our contribution.}

In this paper, we answer the questions raised in~\cite{bonamy2019dominating}. First,  we prove the following:

\begin{theorem}\label{CIG:main}
$\dsr$ is polynomial in circular arc graphs.
\end{theorem}

The very high level idea of the proof is as follows. If we fix a vertex of the dominating set then we can unfold the rest of the graph to get an interval graph. We can then use as a black-box the algorithm of Bonamy et al. on interval graphs to determine if we can slide the fixed vertex of the dominating set to some more desirable position. 

Our second main result is the following:

\begin{theorem}\label{CG:thm:CG}
\dsr is \psc in circle graphs.
\end{theorem}

This is answering a second question of~\cite{bonamy2019dominating}. The proof is inspired from the proof that \textsc{Dominating Set in circle graphs} is NP-complete~\cite{keil1993complexity} but has to be adapted for the reconfiguration framework.

Both our results and the previously known results about the complexity of \dsr in graph classes are summarized in Figure \ref{Intro:fig:classes}.

We left open the following question also raised by Bonamy et al.~\cite{bonamy2019dominating}: does there exist a graph class for which \textsc{Maximum Dominating Set} is NP-complete but \textsc{TS-Reachability} is polynomial? In the reconfiguration world, such results are not frequent but exist. For instance the existence of a reconfiguration sequence between two $3$-colorings can be decided in polynomial time~\cite{CerecedaHJ11} while finding a $3$-coloring is NP-complete. 

We also raise the following question: what is the complexity of the \dsr problem for outerplanar graphs? Outerplanar graphs form a natural subclass of circle graphs, of bounded treewidth graph, and of planar graphs on which the complexity of the problem is PSPACE-complete.

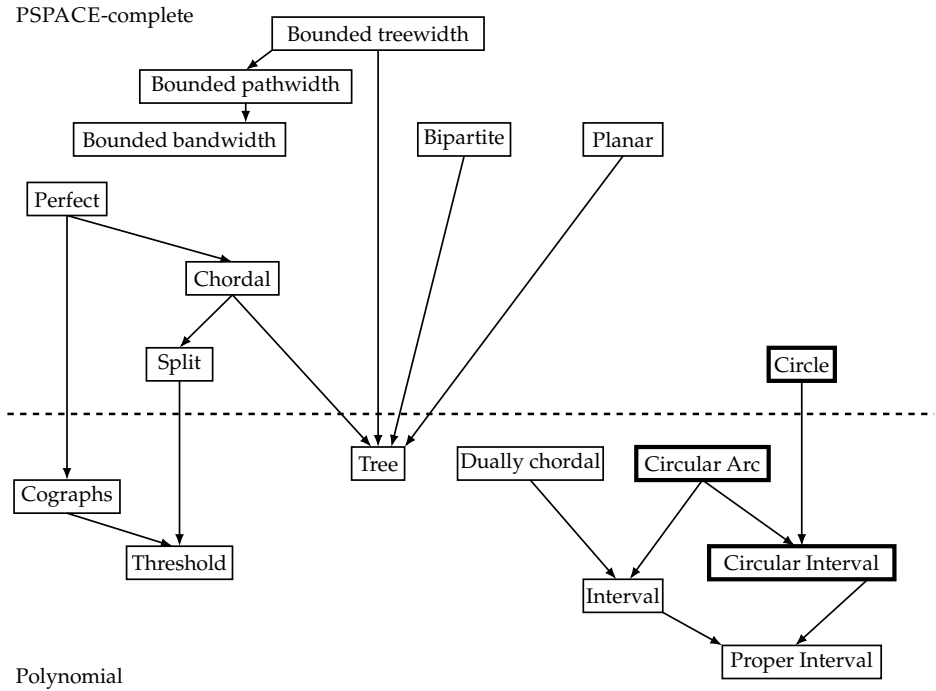
\begin{figure}[ht]
	\centering
	\scalebox{0.8}
	{
	\begin{tikzpicture}[scale=0.11]
    \draw[very thick, dashed] (15,0) -- (155,0);
    \draw[very thick] (15,60) node[right]{\psc};
    \draw[very thick] (15,-40) node[right]{Polynomial};
    
    \draw [line width=0.7mm] (130,5) rectangle (140,10);
    \draw (135,7.5) node{Circle};
    
    \draw [thick] (18,30) rectangle (30,35);
    \draw (24,32.5) node{Perfect};
    
    \draw [thick] (16,-10) rectangle (32,-15);
    \draw (24,-12.5) node{Cographs};
    \draw[thick,-{Latex[length=2mm]}] (24,30) -- (24,-10);
    
    \draw [thick] (42,18) rectangle (56,23);
    \draw (49,20.5) node{Chordal};
    \draw[thick,-{Latex[length=2mm]}] (24,30) -- (49,23);
    
    \draw [thick] (36,5) rectangle (46,10);
    \draw (41,7.5) node{Split};
    \draw[thick,-{Latex[length=2mm]}] (49,18) -- (41,10);
    
    \draw [thick] (33,-20) rectangle (49,-25);
    \draw (41,-22.5) node{Threshold};
    \draw[thick,-{Latex[length=2mm]}] (41,5) -- (41,-20);
    \draw[thick,-{Latex[length=2mm]}] (24,-15) -- (40,-20);
    
    \draw [thick] (67,-5) rectangle (75,-10);
    \draw (71,-7.5) node{Tree};
    
    \draw [thick] (55,55) rectangle (87,60);
    \draw (71,57.5) node{Bounded treewidth};
    
    \draw[thick,-{Latex[length=2mm]}] (71,55) -- (71,-5);
    \draw[thick,-{Latex[length=2mm]}] (49,18) -- (70,-5);
    
    \draw [thick] (35,47) rectangle (67,52);
    \draw (51,49.5) node{Bounded pathwidth};
    
    \draw [thick] (25,39) rectangle (57,44);
    \draw (41,41.5) node{Bounded bandwidth};
    
    \draw[thick,-{Latex[length=2mm]}] (55,55) -- (51,52);
    \draw[thick,-{Latex[length=2mm]}] (51,47) -- (51,44);
    
    \draw [thick] (77,39) rectangle (91,44);
    \draw (84,41.5) node{Bipartite};
    
    
    \draw[thick,-{Latex[length=2mm]}] (84,39) -- (73,-5);
    
    \draw [thick] (102,39) rectangle (114,44);
    \draw (108,41.5) node{Planar};
    \draw[thick,-{Latex[length=2mm]}] (108,39) -- (75,-5);
    
    
    
    \draw [thick] (102,-25) rectangle (114,-30);
    \draw (108,-27.5) node{Interval};
    
    
    \draw [thick] (123,-35) rectangle (147,-40);
    \draw (135,-37.5) node{Proper Interval};
    \draw[thick,-{Latex[length=2mm]}] (114,-30) -- (123,-35);
    \draw[thick,-{Latex[length=2mm]}] (145,-25) -- (134,-35);
    
    \draw [line width=0.7mm] (110,-5) rectangle (130,-10);
    \draw (120,-7.5) node{Circular Arc};
    \draw[thick,-{Latex[length=2mm]}] (120,-10) -- (109,-25);
    
    \draw [line width=0.7mm] (121,-20) rectangle (149,-25);
    \draw (135,-22.5) node{Circular Interval};
    \draw[thick,-{Latex[length=2mm]}] (120,-10) -- (134,-20);
    \draw[thick,-{Latex[length=2mm]}] (135,5) -- (135,-20);
    
    \draw [thick] (83,-5) rectangle (105,-10);
    \draw (94,-7.5) node{Dually chordal};
    \draw[thick,-{Latex[length=2mm]}] (94,-10) -- (107,-25);
    
	\end{tikzpicture}}
    \caption{The complexity of \dsr in several graph classes. The thick rectangles are the results we show in this paper and the other ones are previously known results.} 
	\label{Intro:fig:classes}
\end{figure}

\section{Preliminaries}

Let $G = (V, E)$ be a graph. Given a vertex $v \in V$, $N(v)$ denotes the \emph{open neighborhood} of $v$, i.e. the set $\{ y \in V: vy \in E \}$.

A \emph{multiset} is defined as a set with multiplicities. In other words, in a multiset an element can appear several times. The number of times an element appears is the \emph{multiplicity} of this element. The multiplicity of an element that does not appear in the multiset is $0$. Let $A$ and $B$ be multisets. The \emph{union} of $A$ and $B$, denoted by $A\cup B$, is the multiset containing only elements of $A$ or $B$, and in which the multiplicity of each element is the sum of their multiplicities in $A$ and $B$. The \emph{difference} $A\setminus B$ denotes the multiset containing only elements of $A$, and in which the multiplicity of each element is the difference between its multiplicity in $A$ and its multiplicity in $B$ (if the result is negative then the element is not in $A\setminus B$). By abuse of language, all along this paper, we may refer to multisets as sets. 

A \emph{dominating set} $D$ of $G$ is a multiset of elements of $V$, such that for any $v\in V$, $v\in D$ or there exists $u\in D$ such that $uv\in E$.

Under the token sliding rule, a \emph{move} $v_i\leadsto v_j$, from a set $S_r$ to a set $S_{r+1}$, denotes the token sliding operation along the edge $v_iv_j$ from $v_i$ to $v_j$, i.e. $S_{r+1}=S_r\cup v_j \setminus v_i$. We say that a set $S$ \emph{is before} a set $S'$ in a reconfiguration sequence $\mathcal{S}$ if $\mathcal{S}$ contains a subsequence starting with $S$ and ending with $S'$.

\section{A polynomial time algorithm for circular arc-graphs}\label{sec:Poly}

An \emph{interval graph} $G=(V,E)$ is an intersection graph of intervals of the real line. In other words, the set of vertices is a set of real intervals $I$ and two vertices are adjacent if their corresponding intervals intersect. 
A \emph{circular arc graph} $G=(V,E)$ is an intersection graph of intervals of a circle. In other words, every vertex is associated an arc $A$ of the circle and there is an edge between two vertices if their two corresponding arcs intersect. By abuse of notation, we refer to the vertices by their image arc. Circular arc graphs strictly contain interval graphs (since long induced cycles are circular arc graphs and not interval graphs).
Bonamy et al. proved the following result in~\cite{bonamy2019dominating} that we will use as a black-box:

\begin{theorem}\label{thm:interval}[Bonamy et al.~\cite{bonamy2019conjecture}]
Let $G$ be a connected interval graph, and $D_s,D_t$ be two dominating sets of $G$ of the same size. There always exist a TS-reconfiguration sequence from $D_s$ to $D_t$.
\end{theorem}

One can naturally wonder if Theorem~\ref{thm:interval} can be extended to circular arc graphs. The answer is negative since, for every $k$, the cycle $C_{3k}$ of length $3k$ is a circular arc graph and there are only three dominating sets of size exactly $k$ (the ones containing vertices $i$ mod $3$ for $i \in \{ 0,1,2 \}$) which are pairwise non adjacent for the TS-rule. 

However, we prove that we can decide in polynomial time if we can transform one dominating set into another.The remaining of this section is devoted to prove Theorem~\ref{CIG:main}.

Let $G=(V,E)$ be a circular arc graph and $D_s, D_t$ be two dominating sets of $G$ of the same size $k$. 

Assume first that there exists an arc $v \in V$ that contains the whole circle. So $A$ is a dominating set of $G$ and then for any two dominating sets $D_s$ and $D_t$ of $G$, we can move a token from $D_s$ to $v$, then move every other other token of $D_s$ to a vertex of $D_t$ (in at most two steps passing through $A$), and finally move the token on $v$ to the last vertex of $D_t$. Since a token is on $v$ all along the transformation, all the intermediate steps are indeed dominating sets. So if such an arc exists, there exists a reconfiguration sequence from $D_s$ to $D_t$. 

From now on we assume that no arc contains the whole circle (and that no vertex is dominating the graph).
For any arc $v\in V$, the \emph{left extremity} of $v$, denoted by $\ell(v)$, is the first extremity of $v$ we meet when we follow the circle clockwise, starting from a point outside of $v$. The other extremity of $v$ is called the \emph{right extremity} and is denoted by $r(v)$.We now construct $G_u$ from $G_u'$. First remove the vertex $u$. Note that after this deletion, no arc intersects the open interval $(\ell(u),r(u))$ so the resulting graph is an interval graph. We can unfold it in such a way the first vertex starts at position $\ell(u)$ and the last vertex ends at position $r(u)$ (see Figure~\ref{CIG:fig:CircularLinear}). We add two new vertices, $u'$ and $u''$, that correspond to each extremity of $u$. One has interval $(-\infty,\ell(u)]$ and the other has interval $[r(u),+\infty)$. Since no arc but $u'$ (resp. $u''$) intersects $(-\infty,\ell(u)]$ (resp. $[r(u),+\infty)$), we can create $(n+2)$ new vertices only adjacent to $u'$ (resp. $u''$). These $2n+4$ vertices are called the \emph{leaves} of $G_u$.

Let us first prove the following straightforward lemma.

\begin{lemma}\label{CIG:lem:Inclusion}
Let $G$ be a graph, and $u$ and $v$ be two vertices of $G$ such that $N(u)\subseteq N(v)$. If $S$ is a dominating set reconfiguration sequence in $G$, and $S'$ is obtained from $S$ by replacing every occurrence of $u$ by $v$ in the dominating sets of $S$, then $S'$ also is a dominating set reconfiguration sequence in $G$.
\end{lemma}
\begin{proof}
Every neighbor of $u$ also is a neighbor of $v$. Thus, replacing $u$ by $v$ in a dominating set keeps the domination of $G$. Moreover, any move that involves $u$ can be applied if we replace it by $v$, which gives the result.
\end{proof}

In the proof of Theorem \ref{CIG:main}, we will need the following auxiliary graph $G_u$ (see Figure~\ref{CIG:fig:CircularLinear} for an illustration of the construction). Let $u$ be a vertex of $G$ that is maximal by inclusion (no arc strictly contains it). The circular graph $G_u'$ is the graph such that, for every $v \ne u$ which is not contained in $u$, we create an arc $A_{v'}$ which is the closure of $A_v \setminus A_u$ \footnote{In other words, the arc of $v'$ is the part of the arc of $v$ that is not included in the arc of $u$. Also note that the fact that $A_v'$ is the closure of that arc ensures that that $A_u$ and $A_{v'}$ intersect.}. Since $u$ is maximal by inclusion, $v'$ is an arc. We finally add in $G_u'$ the arc of $u$.
Note that the set of edges in $G_u'$ might be smaller than the one of $G$ but any dominating set of $G$ containing $u$ is a dominating set of $G_u'$. 
We now construct $G_u$ from $G_u'$. First remove the vertex $u$. Note that after this deletion, no arc intersects the open interval $(\ell(u),r(u))$ so the resulting graph is an interval graph. We can unfold it in such a way the first vertex starts at position $\ell(u)$ and the last vertex ends at position $r(u)$ (see Figure~\ref{CIG:fig:CircularLinear}). We add two new vertices, $u'$ and $u''$, that correspond to each extremity of $u$. One has interval $(-\infty,\ell(u)]$ and the other has interval $[r(u),+\infty)$. Since no arc but $u'$ (resp. $u''$) intersects $(-\infty,\ell(u)]$ (resp. $[r(u),+\infty)$), we can create $(n+2)$ new vertices only adjacent to $u'$ (resp. $u''$). These $2n+4$ vertices are called the \emph{leaves} of $G_u$.

\begin{figure}[ht]
	\centering
	\scalebox{0.9}
	{
	\begin{tikzpicture}[scale=0.09]

	\draw (0,0) circle (17);
   \draw [thick,domain=0:100] plot ({cos(\x)*13}, {sin(\x)*13});
   \draw [very thick,domain=70:110] plot ({cos(\x)*11}, {sin(\x)*11});
    \draw [thick,domain=80:100] plot ({cos(\x)*15}, {sin(\x)*15});
   \draw [thick,domain=90:160] plot ({cos(\x)*14}, {sin(\x)*14});
   \draw [thick,domain=130:200] plot ({cos(\x)*15}, {sin(\x)*15});
   \draw [thick,domain=140:210] plot ({cos(\x)*12}, {sin(\x)*12});
   \draw [thick,domain=200:300] plot ({cos(\x)*11}, {sin(\x)*11});
   \draw [thick,domain=210:360] plot ({cos(\x)*12}, {sin(\x)*12});
   \draw [thick,domain=0:30] plot ({cos(\x)*12}, {sin(\x)*12});
   \draw [thick,domain=260:360] plot ({cos(\x)*14}, {sin(\x)*14});
    \draw [thick,domain=0:10] plot ({cos(\x)*14}, {sin(\x)*14});

    \draw[dashed] (0,0) -- ({cos(90)*20}, {sin(90)*20});
    \draw[thick] (2,11) node[below]{$u$};
    \draw (-2,18) node[left]{$-$};
    \draw (2,18) node[right]{$+$};
    
    \draw (30,-2) node[below]{$-$};
    \draw (120,-2) node[below]{$+$};
    \draw (40,0) node[above]{$u'$};
    \draw (109,0) node[above]{$u''$};
    \draw[thick] (38,0) -- (55,0);
    \draw[thick] (50,1) -- (98,1);
    \draw[thick] (60,2) -- (80,2);
    \draw[thick] (90,0) -- (112,0);
    \draw[very thick] (30,1) -- (39,1);
    \draw[very thick] (111,1) -- (120,1);
    \draw[thick] (115,2) -- (119,2);
    \draw[thick] (115,3) -- (119,3);
    \draw[thick] (115,4) -- (119,4);
    \draw[thick] (115,5) -- (119,5);
    \draw[thick] (115,6) -- (119,6);
    \draw[thick] (115,7) -- (119,7);
    \draw[thick] (115,8) -- (119,8);
    \draw[thick] (115,9) -- (119,9);
    \draw[thick] (115,10) -- (119,10);
    \draw[thick] (115,11) -- (119,11);
    \draw[thick] (115,12) -- (119,12);
    \draw[thick] (31,2) -- (35,2);
    \draw[thick] (31,3) -- (35,3);
    \draw[thick] (31,4) -- (35,4);
    \draw[thick] (31,5) -- (35,5);
    \draw[thick] (31,6) -- (35,6);
    \draw[thick] (31,7) -- (35,7);
    \draw[thick] (31,8) -- (35,8);
    \draw[thick] (31,9) -- (35,9);
    \draw[thick] (31,10) -- (35,10);
    \draw[thick] (31,11) -- (35,11);
    \draw[thick] (31,12) -- (35,12);
    \draw[thick] (45,-1) -- (60,-1);
    \draw[thick] (75,-1) -- (92,-1);
    \draw (30,-3) -- (120,-3);
    
    \draw (0,-20) node[below]{$G$};
    \draw (75,-20) node[below]{$G_u$};

	\end{tikzpicture}}
    \caption{The linear interval graph $G_u$ obtained from the circular arc graph $G$.}
	\label{CIG:fig:CircularLinear}
\end{figure}
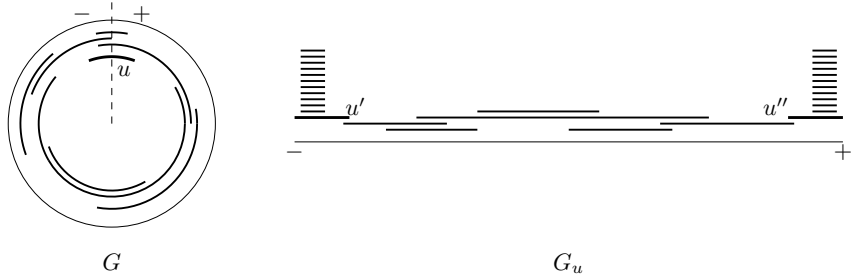

Let us first prove a couple of simple facts about dominating sets of $G_u$.

\begin{lemma}
Let $D$ be a dominating set of $G$ such that $u\in D$, and let $D_u$ be the set $D \cup \{u',u''\} \setminus \{ u \}$. The set $D_u$ is a dominating set of $G_u$.
\end{lemma}
\begin{proof}
Every vertex of $N(u)$ in the original graph $G$ is either not in $G_u$, or is dominated by $u'$ or $u''$. The neighborhood of all the other vertices have not been modified. Moreover, all the new vertices are dominated since they are all adjacent to $u'$ or $u''$. 
\end{proof}

Note that $D_u$ has size $|D|+1$.

\begin{lemma}\label{CIG:lem:Gu2}
The following holds:
\begin{enumerate}
    \item[(i)] All the dominating sets of $G_u$ of size $|D|+1$ contain $u'$ and $u''$.
    \item[(ii)] For every dominating set $X$ of $G_u$ of size $|D|+1$, $(X \cap V) \cup \{ u \}$ is a dominating set of $G$ of size at most $|D|$. 
    \item[(iii)] Every reconfiguration sequence in $G_u$ between two dominating sets $D_s,D_t$ of $G_u$ of size at most $|D|+1$ that does not contain any leaf can be adapted into a reconfiguration sequence in $G$ between $(D_s \setminus \{u',u''\}) \cup \{ u \}$ and $(D_t \setminus \{u',u''\}) \cup \{ u \}$.
\end{enumerate}
\end{lemma}
\begin{proof}
\noindent\textbf{Proof of (i).} 
The point (i) holds since there are $n+2$ leaves attached to each of $u'$ and $u''$ and that $|D|\leq n$.
\smallskip

\noindent\textbf{Proof of (ii).}
 The vertices $u'$ and $u''$ only dominate vertices of $V$ dominated by $u$ in $G$ and $u'$ and $u''$ are in any dominating set of size at most $|D|+1$ of $G_u$ by (i). Moreover no edge between two vertices $x,y \in V(G)$ was created in $G_u$. Thus $(X \cap V) \cup \{ u \}$ is a dominating set of $G$ since the only vertices of $V(G)$ that are not in $V(G_u)$ are vertices whose arcs are strictly included in $u$ and then are dominated by $u$.
\smallskip

\noindent\textbf{Proof of (iii).}
By Lemma~\ref{CIG:lem:Inclusion}, we can assume that there is no token on $u'$ or $u''$ at any point.
We show that we can adapt the transformation. If the move $x \leadsto y$ satisfies that $x,y \notin \{u',u''\}$ then the same edge exists in $G$ and by (ii), the resulting set is dominating. So we can assume that $x$ or $y$ are $u'$ or $u''$. 
We simply have to slide from or to $u$ since $N(u')$ and $N(u'')$ minus the leaves is equal to $N(u)$. Since there is never a token on the leaves, the conclusion follows.
\end{proof}

By Lemma~\ref{CIG:lem:Gu2} and Theorem~\ref{thm:interval}, we immediately obtain the following corollary:

\begin{corollary}\label{coro:circular_u}
Let $G$ be a circular interval graph, $u \in V(G)$, and $k$ be an integer. All the $k$-dominating sets of $G$ containing $u$ are in the same connected component of the reconfiguration graph.
\end{corollary}

We now have all the ingredients to prove Theorem~\ref{CIG:main}.

\begin{proof}[Proof of Theorem~\ref{CIG:main}]
Let $G=(V,E)$ be a circular arc graph, and let $D_s$ and $D_t$ be two dominating sets of $G$. 
Free to slide tokens, we can assume that all the intervals of $D_s$ and $D_t$ are maximal by inclusion. Moreover, by Lemma \ref{CIG:lem:Inclusion}, we can assume that all the vertices of all the dominating sets we will consider are maximal by inclusion. 
By abuse of notation, we say that in $G$, an arc $v$ is the \emph{first arc on the left} (resp. \emph{on the right}) of another arc $u$ if the first left extremity of an inclusion-wise maximal arc (of $G$, or of the stated dominating set) we encounter when browsing the circle counter clockwise (resp. clockwise) from the left extremity of $u$ is the one of $v$. 
In interval graphs, we say that an interval $v$ is at \emph{the left} (resp. at \emph{the right}) of an interval $u$ if the left extremity of $v$ is smaller (resp. larger) than the one of $u$. Note that since the intervals of the dominating sets are maximal by inclusion, the left and right ordering of these vertices are the same. So we can assume that we have a total ordering of the vertices of the dominating sets we are considering.
\smallskip

Let $u_1$ be a vertex of $D_s$. Let $v$ be the first vertex at the right of $u_1$ in $D_t$. We perform the following algorithm, called the Right Sliding Algorithm. By Lemma~\ref{CIG:lem:Gu2}, all the dominating sets of size $|D_s|+1$ in $G_{u_1}$ contain $u_1'$ and $u_1''$. Let $D_2'$ be a dominating set of the interval graph $G_{u_1}$ of size $|D_s|+1$, such that the first vertex at the right of $u_1'$ has the smallest left extremity (we can indeed find such a dominating set in polynomial time). By Theorem~\ref{thm:interval}, there exists a transformation from $(D_s \cup \{ u_1',u_1'' \}) \setminus \{ u_1 \}$ to $D_2'$ in $G_{u_1}$. And thus by Lemma~\ref{CIG:lem:Gu2}, there exists a transformation from $D_s$ to $D_2:= (D_2' \cup \{ u_1 \}) \setminus \{ u_1',u_1''\}$ in $G$. We apply this transformation. Informally speaking, this transformation has permitted to move the token at the left of $u$ closest from $u$ that will hopefully permit to push the token on $u$ to the right.

Now, we fix all the vertices of $D_2$ but $u_1$, and we try to slide the token on $u_1$ to its right. If we can push it on a vertex at the right of $v$, we can in particular push it on $v$ (since $v$ is maximal by inclusion) and keep a dominating set. So we set $u_2=v$ if we can reach $v$ or the rightmost possible vertex maximal by inclusion we can reach otherwise.
We now repeat these operations with $u_2$ instead of $u_1$, i.e. we apply a reconfiguration sequence towards a dominating set of $G$ in which the first vertex on the left of $u_2$ is the closest to $u_2$, then try to slide $u_2$ to the right, onto $u_3$. We repeat these operations until $u_i=u_{i+1}$ (i.e. we cannot move to the right anymore) or until $u_i=v$. Let $u_1,\ldots,u_\ell$ be the resulting sequence of vertices. Note that this algorithm is indeed polynomial since after at most $n$ steps we have reached $v$ or reached a fixed point.

We can similarly define the Left Sliding Algorithm by replacing the leftmost dominating set of $G_{u_i}$ by the rightmost, and then slide $u_i$ to the left for any $i$. We stop when we cannot slide to the left anymore, or when $u_i=v'$, where $v'$ is the first vertex at the left of $u_1$ in $D_t$. Let $u_\ell '$ be the last vertex of the sequence of vertices given by the Left Sliding Algorithm.

To conclude the proof we simply have to show the following claim: 

\begin{claim}
There exists a transformation from $D_s$ to $D_t$ if and only if $u_\ell=v$ or $u_\ell'=v'$.
\end{claim}
\begin{proofclaim}
Firstly, if $u_\ell=v$, then Corollary \ref{coro:circular_u} ensures that there exists a transformation from $D_\ell$ to $D_t$ and thus from $D_s$ to $D_t$, and similarly if $u_\ell'=v'$. 

Let us now prove the converse direction. If $u_\ell\neq v$ and $u_\ell' \neq v'$, assume for contradiction that there exists a transformation sequence $S$ from $D_s$ to $D_t$. By Lemma~\ref{CIG:lem:Inclusion} we can assume that all the vertices in any dominating set of $S$ are maximal by inclusion. 

Let us consider the first dominating set $C$ of $S$ where the token initially on $u_1$ is at the right of $u_\ell$ in $G$, or at the left of $u_\ell'$ in $G$. Such a dominating set exists no token of $D_t$ is between $u'_{\ell'}$ and $u_\ell$. Let us denote by $C'$ the dominating before $C$ in the sequence and $x \leadsto y$ the move from $C$ to $C'$.
By symmetry, we can assume that $y$ is at the right of $u_\ell$. Note that $x$ is at the left of $u_\ell$. 
Note that $C''=C \setminus \{ x \} \cup \{u_\ell \}$ is a dominating set of $G$ since $C$ and $C'=C \setminus \{ x \} \cup \{y \}$ are dominating sets and $u_\ell$ is between $x$ and $y$.

So $C \setminus \{ x\} \cup \{u_\ell,u_\ell'\}$ is a dominating set of $G_{u_\ell}$ and then for $C''$ it was possible to move the token on $u_\ell$ to the right, a contradiction with the fact that $u_\ell$ was a fixed point. 
\end{proofclaim}
\end{proof}

\section{PSPACE-hardness for Circle Graphs}\label{sec:CircleGraphs}

A \emph{circle graph} $G=(V,E)$ is an intersection graph of chords of a circle (i.e. segments between two points of a circle). Let $C$ be a circle. Equivalently, we can associate to each vertex of  a circle graph two points of $C$. And there is an edge between two vertices if the chords between their pair of points intersect. Again equivalently, a circle graph can be represented on the real line. We associate to each vertex an interval of the real line; and there is an edge between two vertices if their intervals intersect but do not overlap. In this section, we will use the last representation of circle graphs. For every interval $I$, $\ell(I)$ will denote the \emph{left extremity} of $I$, and $r(I)$ the \emph{right extremity} of $I$.

The goal of this section is to show that \dsr is \psc in circle graphs. We provide a polynomial time reduction from \satr to \dsrnospace. This reduction is inspired from one used in~\cite{keil1993complexity} to show that the minimum dominating set problem is \npc on circle graphs but has to be adapted in the reconfiguration framework.
The \satr problem is defined as follows:
\smallskip

\noindent \satrfull (\satr)\\
\textbf{Input:} A Boolean formula $F$ in conjunctive normal form (conjunction of clauses), two variable assignments $A_s$ and $A_t$ that satisfy $F$.\\
\textbf{Output:} Does there exist a reconfiguration sequence from $A_s$ to $A_t$ that keeps $F$ satisfied, where the operation consists in a \emph{variable flip}, i.e. the change of the assignment of exactly one variable from $x=0$ to $x=1$, or conversely ?
\smallskip

Let $(F,A_s,A_t)$ be an instance of the \satr problem. Let $x_1\,\ldots,x_n$ be the variables of the boolean formula $F$. Since $F$ is in conjunctive normal form, it is a conjunction of \emph{clauses} $c_1,\ldots,c_m$ which are disjunctions of literals. A \emph{literal} is a variable or the negation of a variable, and we denote by $x_i\in c_j$ (resp. $\overline{x_i}\in c_j$) the fact that $x_i$ (resp. the negation of $x_i$) is a literal of $c_j$. Since duplicating clauses does not modify the satisfiability of a formula, we can assume without loss of generality that $m$ is a multiple of $4$. We can also assume that for every $i \le n$ and $j \le m$, and that, for every $i,j$, $x_i$ or $\overline{x_i}$ are not in $c_j$ (since otherwise the clause is satisfied for any possible assignment and can be removed from the boolean formula).

\subsection{The reduction.}
Let us construct an instance $(G_F,D_F(A_s),D_F(A_t))$ of the \dsr problem from $(F,A_s,A_t)$. We start by constructing the circle graph $G_F$ from $F$. All along this construction, we repeatedly refer to real number as \emph{points}. We say that a point $p$ is \emph{at the left} of a point $q$ (or $q$ is \emph{at the right} of $p$) if $p<q$. We say that $p$ is \emph{just at the left} of $q$, (or $q$ is \emph{just at the right} of $p$) if $p$ is at the left of $q$, and no interval defined so far has an extremity in $[p,q]$. 
Finally, we say that an interval $I$ \emph{frames} a set of points $P$ if $\ell(I)$ is just at the left of the minimum of $P$ and $r(I)$ is just at the right of the maximum of $P$.

One can easily check that by adding an interval that frames one extremity of the interval of a vertex $u$ of a graph $H$, we add one vertex to $H$ which is only connected to $u$. So:

\begin{remark}\label{rk:AdjInterval}
If $H$ is a circle graph and $u$ is a vertex of $H$, then the graph $H$ plus a new vertex only connected to $u$ is circle graph.
\end{remark}


We construct $G_F$ step by step. The construction of $G_F$ is quite technical and will be performed step by step. The construction is inspired from~\cite{keil1993complexity}. In~\cite{keil1993complexity}, the authors have decided to give the coordinates of the endpoints of all the intervals. For the sake of readability, we think that it is easier to only give the relative positions of the intervals between them. 

Each step consists in creating new intervals, and in giving their relative positions regarding to the previously constructed intervals. We also outline some of the edges and non edges in $G_F$ that have an impact on the upcoming proofs\footnote{Some adjacencies between intervals that will be anyway dominated for some reasons that will become clear later on will not be discussed.}. Figures \ref{CG:fig:varintervals}, \ref{CG:fig:allintervals} and \ref{CG:fig:allintervalszoom} will illustrate the positions of the intervals of $G_F$.

For each variable $x_i$, we create $m$ \emph{base intervals} $B_j^i$ where $1\leq j\leq m$. The base intervals $B_j^i$ are pairwise disjoint for any $i$ and $j$, and are ordered by increasing $i$, then increasing $j$ for a same $i$. 

For each variable $x_i$, we then create $\frac{m}{2}$ intervals $X_j^i$ called the \emph{positive bridge intervals} of $x_i$, and $\frac{m}{2}$ intervals $\overline{X}_j^i$ called the \emph{negative bridge intervals} of $x_i$, where $1\leq j\leq \frac{m}{2}$. A \emph{bridge interval} is a positive or a negative bridge interval. Let us give the positions of these intervals. They are illustrated in Figure \ref{CG:fig:varintervals}.

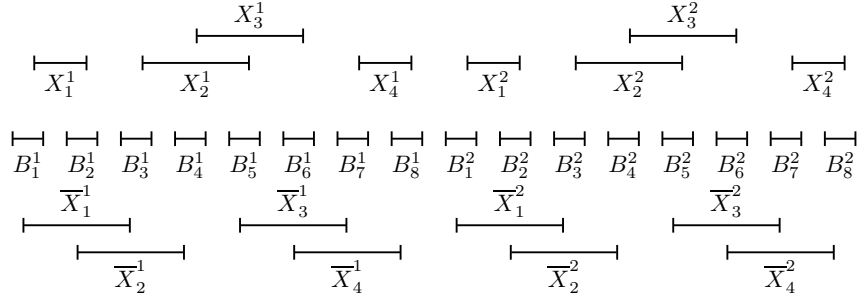
\begin{figure}[ht]
	\centering
	\scalebox{0.9}
	{
	\begin{tikzpicture}[scale=0.04]
    
    \draw (5,0) node[below]{$B_1^1$};
    \draw[|-|,thick] (-1,2) -- (11,2);
    \draw (25,0) node[below]{$B_2^1$};
    \draw[|-|,thick] (19,2) -- (31,2);
    \draw (45,0) node[below]{$B_3^1$};
    \draw[|-|,thick] (39,2) -- (51,2);
    \draw (65,0) node[below]{$B_4^1$};
    \draw[|-|,thick] (59,2) -- (71,2);
    \draw (85,0) node[below]{$B_5^1$};
    \draw[|-|,thick] (79,2) -- (91,2);
    \draw (105,0) node[below]{$B_6^1$};
    \draw[|-|,thick] (99,2) -- (111,2);
    \draw (125,0) node[below]{$B_7^1$};
    \draw[|-|,thick] (119,2) -- (131,2);
    \draw (145,0) node[below]{$B_8^1$};
    \draw[|-|,thick] (139,2) -- (151,2);
    \draw (165,0) node[below]{$B_1^2$};
    \draw[|-|,thick] (159,2) -- (171,2);
    \draw (185,0) node[below]{$B_2^2$};
    \draw[|-|,thick] (179,2) -- (191,2);
    \draw (205,0) node[below]{$B_3^2$};
    \draw[|-|,thick] (199,2) -- (211,2);
    \draw (225,0) node[below]{$B_4^2$};
    \draw[|-|,thick] (219,2) -- (231,2);
    \draw (245,0) node[below]{$B_5^2$};
    \draw[|-|,thick] (239,2) -- (251,2);
    \draw (265,0) node[below]{$B_6^2$};
    \draw[|-|,thick] (259,2) -- (271,2);
    \draw (285,0) node[below]{$B_7^2$};
    \draw[|-|,thick] (279,2) -- (291,2);
    \draw (305,0) node[below]{$B_8^2$};
    \draw[|-|,thick] (299,2) -- (311,2);
    
    \draw (23,-30) node[above]{$\overline{X}_1^1$};
    \draw[|-|,thick] (3,-30) -- (43,-30);
    \draw (103,-30) node[above]{$\overline{X}_3^1$};
    \draw[|-|,thick] (83,-30) -- (123,-30);
    \draw (183,-30) node[above]{$\overline{X}_1^2$};
    \draw[|-|,thick] (163,-30) -- (203,-30);
    \draw (263,-30) node[above]{$\overline{X}_3^2$};
    \draw[|-|,thick] (243,-30) -- (283,-30);
    
    \draw (43,-40) node[below]{$\overline{X}_2^1$};
    \draw[|-|,thick] (23,-40) -- (63,-40);
    \draw (123,-40) node[below]{$\overline{X}_4^1$};
    \draw[|-|,thick] (103,-40) -- (143,-40);
    \draw (203,-40) node[below]{$\overline{X}_2^2$};
    \draw[|-|,thick] (183,-40) -- (223,-40);
    \draw (283,-40) node[below]{$\overline{X}_4^2$};
    \draw[|-|,thick] (263,-40) -- (303,-40);
    
     \draw (17,30) node[below]{$X_1^1$};
    \draw[|-|,thick] (7,30) -- (27,30);
    \draw (137,30) node[below]{$X_4^1$};
    \draw[|-|,thick] (127,30) -- (147,30);
    \draw (177,30) node[below]{$X_1^2$};
    \draw[|-|,thick] (167,30) -- (187,30);
    \draw (297,30) node[below]{$X_4^2$};
    \draw[|-|,thick] (287,30) -- (307,30);
    
     \draw (67,30) node[below]{$X_2^1$};
    \draw[|-|,thick] (47,30) -- (87,30);
     \draw (227,30) node[below]{$X_2^2$};
    \draw[|-|,thick] (207,30) -- (247,30);

    \draw (87,40) node[above]{$X_3^1$};
    \draw[|-|,thick] (67,40) -- (107,40);
     \draw (247,40) node[above]{$X_3^2$};
    \draw[|-|,thick] (227,40) -- (267,40);
	\end{tikzpicture}}
    \caption{The base, positive and negative bridge intervals obtained with $n=2$ and $m=8$.}
	\label{CG:fig:varintervals}
\end{figure}

Let $q$ be such that $m=4q$. For every $i$ and every $0 \le r < q$, the interval $\overline{X}_{2r+1}^i$ starts just at the right of $\ell(B_{4r+1}^i)$ and ends just at the right of $\ell(B_{4r+3}^i)$, and $\overline{X}_{2r+2}^i$ starts just at the right of $\ell(B_{4r+2}^i)$ and ends just at the right of $\ell(B_{4r+4}^i)$. The interval $X_1^i$ starts just at the left of $r(B_1^i)$ and ends just at the left of $r(B_2^i)$. For every $1 \le r < q$, the interval $X_{2r}^i$ starts just at the left of $r(B_{4r-1}^i)$ and ends just at the left of $r(B_{4r+1}^i)$, and $X_{2r+1}^i$ starts just at the left of $r(B_{4r}^i)$ and ends just at the left of $r(B_{4r+2}^i)$. Finally, $X_{\frac{m}{2}}^i$ starts just at the left of $r(B_{m-1}^i)$ and ends just at the left of $r(B_m^i)$.




Let us outline some of the edges induced by these intervals. 
Base intervals are pairwise non adjacent.
Moreover, every positive (resp. negative) bridge interval is incident to exactly two base intervals; And all the positive (resp. negative) bridge intervals of $x_i$ are incident to pairwise distinct base intervals. In particular, the positive (resp. negative) bridge intervals dominate the base intervals; And every base interval is adjacent to exactly one positive and one negative bridge interval. 
All the positive (resp. negative) bridge intervals but $X_1^i$ and $X_{\frac{m}{2}}^i$ have exactly one other positive (resp. negative) bridge interval neighbor. Finally, for every $i$, every negative bridge interval $\overline{X}_j^i$ has exactly two positive bridge interval neighbors which are $X_{j-1}^i$ and $X_{j}^i$ except for $\overline{X}_1^i$ which does not have any for any $i$. Note that a bridge interval of $x_i$ is not adjacent to a bridge interval or a base interval of $x_j$ for $j \ne i$.
\smallskip

Now for any clause $c_j$, we create two identical \emph{clause intervals} $C_j$ and $C_j'$. In this paper, we consider that two identical intervals do overlap, so that $C_j$ and $C_j'$ are not adjacent. The clause intervals $C_j$ are pairwise disjoint and ordered by increasing $j$, and we have $\ell(C_1)>r(B_m^n)$. Thus, they are not adjacent to any interval constructed so far.

For every $j$ such that $x_i$ is in the clause $c_j$, we create four intervals  $T_j^i$, $U_j^i$, $V_j^i$ and $W_j^i$, called the \emph{positive path intervals} of $x_i$; and for every $j$ such that $\overline{x_i}$ is in the clause $c_j$, we create four intervals $\overline{T}_j^i$, $\overline{U}_j^i$, $\overline{V}_j^i$ and $\overline{W}_j^i$, called the \emph{negative path intervals} of $x_i$. These intervals are represented in Figure \ref{CG:fig:allintervals}. In order to give a better representation of the relative position of the extremities, a zoom on that part of the graph is proposed in Figure \ref{CG:fig:allintervalszoom}. The interval $T_j^i$ frames the right extremity of $B_j^i$ and the extremity of the positive bridge interval that belongs to $B_j^i$. The interval $\overline{T}_j^i$ frames the left extremity of $B_j^i$ and the extremity of the negative bridge interval that belongs to $B_j^i$. The interval $U_j^i$ starts just at the left of $r(T_j^i)$, the interval $\overline{U}_j^i$ starts just at the right of $l(\overline{T}_j^i)$, and they both end between the right of the last base interval of the variable $x_i$ and the left of the next base or clause interval. We moreover construct the intervals $U_j^i$ (resp. $\overline{U}_j^i$) in such a way $r(U_j^i)$ (resp. $r(\overline{U}_j^i)$) is increasing when $j$ is increasing. In other words, the $U_j^i$ (resp. $\overline{U}_j^i$) are pairwise adjacent. The interval $V_j^i$ (resp. $\overline{V_j^i}$) frames the right extremity of $U_j^i$ (resp. $\overline{U_j^i}$). And the interval $W_j^i$ (resp. $\overline{W_j^i}$) starts just at the left of $r(V_j^i)$ (resp. $r(\overline{V_j^i})$) and ends in an arbitrary point of $C_j$. Moreover, for any $i\neq i'$, $W_j^i$ (resp. $\overline{W_j^i}$) and $W_j^{i'}$ (resp. $\overline{W_j^{i'}}$) end on the same point of $C_j$. This ensures that they overlap and are therefore not adjacent.



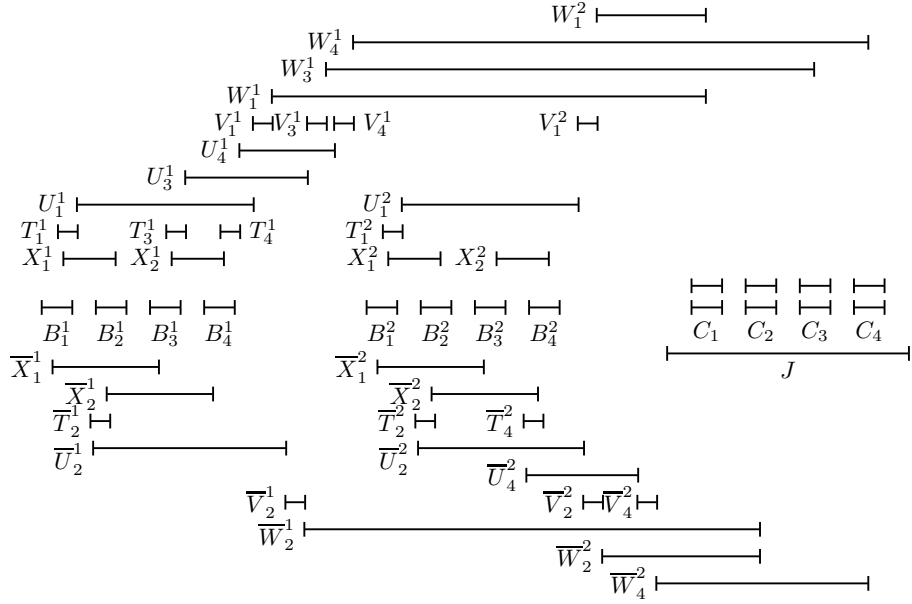
\begin{figure}[ht]
	\centering
	\scalebox{0.9}
	{
	\begin{tikzpicture}[scale=0.04]
    
    \draw (5,0) node[below]{$B_1^1$};
    \draw[|-|,thick] (-1,2) -- (11,2);
    \draw (25,0) node[below]{$B_2^1$};
    \draw[|-|,thick] (19,2) -- (31,2);
    \draw (45,0) node[below]{$B_3^1$};
    \draw[|-|,thick] (39,2) -- (51,2);
    \draw (65,0) node[below]{$B_4^1$};
    \draw[|-|,thick] (59,2) -- (71,2);
    \draw (125,0) node[below]{$B_1^2$};
    \draw[|-|,thick] (119,2) -- (131,2);
    \draw (145,0) node[below]{$B_2^2$};
    \draw[|-|,thick] (139,2) -- (151,2);
    \draw (165,0) node[below]{$B_3^2$};
    \draw[|-|,thick] (159,2) -- (171,2);
    \draw (185,0) node[below]{$B_4^2$};
    \draw[|-|,thick] (179,2) -- (191,2);
    \draw (245,0) node[below]{$C_1$};
    \draw[|-|,thick] (239,2) -- (251,2);
    \draw (265,0) node[below]{$C_2$};
    \draw[|-|,thick] (259,2) -- (271,2);
    \draw (285,0) node[below]{$C_3$};
    \draw[|-|,thick] (279,2) -- (291,2);
    \draw (305,0) node[below]{$C_4$};
    \draw[|-|,thick] (299,2) -- (311,2);
    \draw[|-|,thick] (239,10) -- (251,10);
    \draw[|-|,thick] (259,10) -- (271,10);
    \draw[|-|,thick] (279,10) -- (291,10);
    \draw[|-|,thick] (299,10) -- (311,10);
    \draw[|-|,thick] (230,-15) -- (320,-15);
    \draw (275,-15) node[below]{$J$};

    \draw (3,-20) node[left]{$\overline{X}_1^1$};
    \draw[|-|,thick] (3,-20) -- (43,-20);
    \draw (123,-20) node[left]{$\overline{X}_1^2$};
    \draw[|-|,thick] (123,-20) -- (163,-20);
    \draw (23,-30) node[left]{$\overline{X}_2^1$};
    \draw[|-|,thick] (23,-30) -- (63,-30);
    \draw (143,-30) node[left]{$\overline{X}_2^2$};
    \draw[|-|,thick] (143,-30) -- (183,-30);
    
     \draw (7,20) node[left]{$X_1^1$};
    \draw[|-|,thick] (7,20) -- (27,20);
    \draw (47,20) node[left]{$X_2^1$};
    \draw[|-|,thick] (47,20) -- (67,20);
    \draw (127,20) node[left]{$X_1^2$};
    \draw[|-|,thick] (127,20) -- (147,20);
    \draw (167,20) node[left]{$X_2^2$};
    \draw[|-|,thick] (167,20) -- (187,20);

    \draw (5,30) node[left]{$T_1^1$};
    \draw[|-|,thick] (5,30) -- (13,30);
    \draw (45,30) node[left]{$T_3^1$};
    \draw[|-|,thick] (45,30) -- (53,30);
    \draw (73,30) node[right]{$T_4^1$};
    \draw[|-|,thick] (65,30) -- (73,30);
    \draw (125,30) node[left]{$T_1^2$};
    \draw[|-|,thick] (125,30) -- (133,30);
    
    \draw (17,-40) node[left]{$\overline{T}_2^1$};
    \draw[|-|,thick] (17,-40) -- (25,-40);
    \draw (137,-40) node[left]{$\overline{T}_2^2$};
    \draw[|-|,thick] (137,-40) -- (145,-40);
    \draw (177,-40) node[left]{$\overline{T}_4^2$};
    \draw[|-|,thick] (177,-40) -- (185,-40);
    
    \draw (12,40) node[left]{$U_1^1$};
    \draw[|-|,thick] (12,40) -- (78,40);
    \draw (52,50) node[left]{$U_3^1$};
    \draw[|-|,thick] (52,50) -- (98,50);
    \draw (72,60) node[left]{$U_4^1$};
    \draw[|-|,thick] (72,60) -- (108,60);
    \draw (132,40) node[left]{$U_1^2$};
    \draw[|-|,thick] (132,40) -- (198,40);
    
    \draw (18,-54) node[left]{$\overline{U}_2^1$};
    \draw[|-|,thick] (18,-50) -- (90,-50);
    \draw (138,-54) node[left]{$\overline{U}_2^2$};
    \draw[|-|,thick] (138,-50) -- (200,-50);
    \draw (178,-60) node[left]{$\overline{U}_4^2$};
    \draw[|-|,thick] (178,-60) -- (220,-60);
    
    \draw (77,70) node[left]{$V_1^1$};
    \draw[|-|,thick] (77,70) -- (85,70);
    \draw (99,70) node[left]{$V_3^1$};
    \draw[|-|,thick] (97,70) -- (105,70);
    \draw (115,70) node[right]{$V_4^1$};
    \draw[|-|,thick] (107,70) -- (115,70);
    \draw (197,70) node[left]{$V_1^2$};
    \draw[|-|,thick] (197,70) -- (205,70);
    
    \draw (89,-70) node[left]{$\overline{V}_2^1$};
    \draw[|-|,thick] (89,-70) -- (97,-70);
    \draw (199,-70) node[left]{$\overline{V}_2^2$};
    \draw[|-|,thick] (199,-70) -- (207,-70);
    \draw (221,-70) node[left]{$\overline{V}_4^2$};
    \draw[|-|,thick] (219,-70) -- (227,-70);
    
    \draw (84,80) node[left]{$W_1^1$};
    \draw[|-|,thick] (84,80) -- (245,80);
    \draw (104,90) node[left]{$W_3^1$};
    \draw[|-|,thick] (104,90) -- (285,90);
    \draw (114,100) node[left]{$W_4^1$};
    \draw[|-|,thick] (114,100) -- (305,100);
    \draw (204,110) node[left]{$W_1^2$};
    \draw[|-|,thick] (204,110) -- (245,110);
    
    \draw (96,-82) node[left]{$\overline{W}_2^1$};
    \draw[|-|,thick] (96,-80) -- (265,-80);
    \draw (206,-90) node[left]{$\overline{W}_2^2$};
    \draw[|-|,thick] (206,-90) -- (265,-90);
    \draw (226,-100) node[left]{$\overline{W}_4^2$};
    \draw[|-|,thick] (226,-100) -- (305,-100);
	\end{tikzpicture}}
    \caption{The intervals obtained for the formula $F=(x_1\lor x_2)\land(\overline{x_1}\lor \overline{x_2})\land (x_1)\land (\overline{x_2}\lor x_1)$ with $m=4$ clauses and $n=2$ variables. The dead-end intervals and the pending intervals are not represented here.}
	\label{CG:fig:allintervals}
\end{figure}

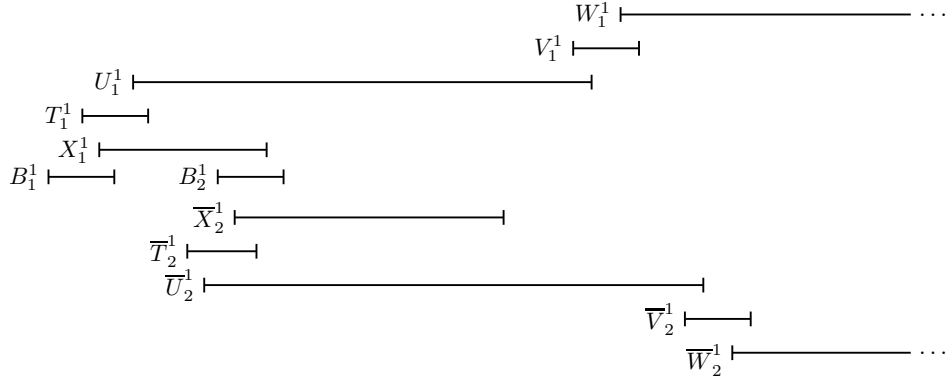
\begin{figure}[ht]
	\centering
	\scalebox{0.9}
	{
	\begin{tikzpicture}[scale=0.05]
    
    \draw (-5,2) node[left]{$B_1^1$};
    \draw[|-|,thick] (-5,2) -- (15,2);
    \draw (45,2) node[left]{$B_2^1$};
    \draw[|-|,thick] (45,2) -- (65,2);
    
    \draw (50,-10) node[left]{$\overline{X}_2^1$};
    \draw[|-|,thick] (50,-10) -- (130,-10););
    
     \draw (10,10) node[left]{$X_1^1$};
    \draw[|-|,thick] (10,10) -- (60,10);

    \draw (5,20) node[left]{$T_1^1$};
    \draw[|-|,thick] (5,20) -- (25,20);
    
    \draw (36,-20) node[left]{$\overline{T}_2^1$};
    \draw[|-|,thick] (36,-20) -- (57,-20);
    
    \draw (20,30) node[left]{$U_1^1$};
    \draw[|-|,thick] (20,30) -- (156,30);
    
    \draw (41,-30) node[left]{$\overline{U}_2^1$};
    \draw[|-|,thick] (41,-30) -- (189,-30);
    
    \draw (150,40) node[left]{$V_1^1$};
    \draw[|-|,thick] (150,40) -- (170,40);
    
    \draw (183,-40) node[left]{$\overline{V}_2^1$};
    \draw[|-|,thick] (183,-40) -- (203,-40);
    
    \draw (164,50) node[left]{$W_1^1$};
    \draw[|-,thick] (164,50) -- (250,50);
    \draw (250,50) node[right]{$\ldots$};
    
    \draw (197,-52) node[left]{$\overline{W}_2^1$};
    \draw[|-,thick] (197,-50) -- (250,-50);
    \draw (250,-50) node[right]{$\ldots$};

	\end{tikzpicture}}
    \caption{A zoom on some intervals of the variable $x_1$.}
	\label{CG:fig:allintervalszoom}
\end{figure}


A \emph{path interval} is a positive or a negative path interval. The \emph{intervals of $x_i$} are the base, bridge and path intervals of $x_i$. The $T$ intervals of $x_i$ refers to the intervals $T_j^i$ for any $j$. The $\overline{T}$, $U$, $\overline{U}$, $V$, $\overline{V}$, $W$ and $\overline{W}$ intervals of $x_i$ are defined similarly.

Let us outline some neighbors of the path intervals. The neighborhood of every clause interval $C_j$ is the set of intervals $W_j^i$ with $x_i\in c_j$ and intervals $\overline{W}^i_j$ with $\overline{x_i}\in c_j$. 
Since $V_j^i$ spans the left extremity of $W_j^i$ and the right extremity of $U_j^i$ and since no interval starts or ends between these two points, the interval $V_j^i$ is only adjacent to $U_j^i$ and $W_j^i$. Similarly  $\overline{V}_j^i$ is only adjacent to $\overline{U}_j^i$ and $\overline{W}_j^i$. Moreover, $T_j^i$ is only adjacent to $B_j^i$, $U_j^i$ and one positive bridge interval (the same one that is adjacent to $B_j^i$), and $\overline{T}_j^i$ is only adjacent to $B_j^i$, $\overline{U}_j^i$ and one negative bridge interval (the same one that is adjacent to $B_j^i$).
Moreover, since $U_j^i$ and $W_j^i$ are not adjacent, $B_j^i$, $T_j^i$, $U_j^i$, $V_j^i$, $W_j^i$ and $C_j$ induce a path, and since $\overline{U}_j^i$ and $\overline{W}_j^i$ are not adjacent, $B_j^i$, $\overline{T}_j^i$, $\overline{U}_j^i$, $\overline{V}_j^i$, $\overline{W}_j^i$ and $C_j$ induce a path. Finally, for any two variables $x_i$ and $x_i'$ such that $x_i\neq x_i'$, the only path intervals of respectively $x_i$ and $x_i'$ that can be adjacent are the $W$ and $\overline{W}$ intervals adjacent to different clause intervals.

Now, for every bridge interval and every $U$, $\overline{U}$, $W$ and $\overline{W}$ interval, we create a \emph{dead-end interval}, that is only adjacent to it. Remark \ref{rk:AdjInterval} ensures that it can be done while keeping a circle graph. Then, for any dead-end interval, we create $6mn$ \emph{pending intervals} that are each only adjacent to it. Again, Remark \ref{rk:AdjInterval} ensures that the resulting graph is a circle graph. Informally speaking, since the dead-end intervals have a lot of pending intervals, they will be forced to be in any dominating set of size at most $6mn$. Thus, in any dominating set, we will know that bridge, $U$, $\overline{U}$, $W$ and $\overline{W}$ intervals (as well as dead-end an pending vertices) are already dominated. So the other vertices in the dominating set will only be there to dominate the other vertices of the graph, which are called the \emph{important vertices}.

Finally, we create a \emph{junction interval} $J$, that frames $\ell(C_1)$ and $r(C_m)$. By construction, it is adjacent to every $W$ or $\overline{W}$ interval, and to no other interval.  This completes the construction of the graph $G_F$.

\subsection{Basic properties of $G_F$} 

Let us first give a couple of properties satisfied by $G_F$. The following lemma will be used to guarantee that any token can be moved to any vertex of the graph as long as the rest of the tokens form a dominating set.

\begin{lemma}\label{CG:lem:connected}
The graph $G_F$ is connected.
\end{lemma}
\begin{proof}
Let $x_i$ be a variable. Let us first prove that the intervals of $x_i$ are in the same connected component of $G_F$. (Recall that they are the base, bridge and path intervals of $x_i$). Firstly, for any $j$ such that $x_i\in C_j$ (resp. $\overline{x_i} \in C_j$), $B_j^i T_j^i U_j^i V_j^i W_j^i$ (resp.  $B_j^i \overline{T}_j^i \overline{U}_j^i \overline{V}_j^i \overline{W}_j^i$) is a path of $G_F$. Since every base interval of $x_i$ is adjacent to a positive and a negative bridge interval of $x_i$, it is enough to show that all the bridge intervals of $x_i$ are in the same connected component. Since for every $j \ge 2$, $\overline{X}_j^i$ is adjacent to $X_{j-1}^i$ and $X_{j}^i$, we know that $X_1^i \overline{X}_2^i,X_2^i \ldots \overline{X}_{\frac{m}{2}}^i X_{\frac{m}{2}}^i$ is a path of $G_F$. Moreover, $\overline{X}_1^i$ is adjacent to $\overline{X}_2^i$. So all the intervals of $x_i$ are in the same connected component of $G_F$..

Now, since the junction interval $J$ is adjacent to every $W$ and $\overline{W}$ interval (and that each variable appears in at least one clause), $J$ is in the connected component of all the path variables, so the intervals of $x_i$ and $x_{i'}$ are in the same connected component for every $i \ne i'$. Since each clause contains at least one variable, $C_j$ is adjacent to at least one interval $W_j^i$ or $\overline{W}_j^i$. Finally, each dead-end interval is adjacent to a bridge interval or a $U$, $\overline{U}$, $W$ or $\overline{W}$ interval, and each pendant interval is adjacent to a dead-end interval. Therefore, $G_F$ is connected.
\end{proof}

For any variable assignment $A$ of $F$, let $D_F(A)$ be the set of intervals of $G_F$ defined as follows. The junction interval $J$ belongs to $D_F(A)$ and all the dead-end intervals belong to $D_F(A)$. For any variable $x_i$ such that $x_i=1$ in $A$, the positive bridge, $W$ and $\overline{U}$ intervals of $x_i$ belong to $D_F(A)$. Finally, for any variable $x_i$ such that $x_i=0$ in $A$, the negative bridge, $\overline{W}$ and $U$ intervals of $x_i$ belong to $D_F(A)$. The multiplicity of each of these intervals in $D_F(A)$ is one. Thus, we have $|D_F(A)|=\frac{3mn}{2} + 3\sum_{i=1}^n \ell_i + 1$ where for any variable $x_i$, $\ell_i$ is the number of clauses that contain $x_i$ or $\overline{x_i}$.

\begin{lemma}\label{CG:lem:DS}
If $A$ satisfies $F$, then $D_F(A)\setminus J$ is a dominating set of $G_F$.
\end{lemma}
\begin{proof}
Since every dead-end interval belongs to $D_F(A)\setminus J$, every pending and dead-end interval is dominated, as well as every bridge, $U$, $\overline{U}$, $W$ and $\overline{W}$ interval. 
Since for each variable $x_i$, the positive (resp. negative) bridge intervals of $x_i$ dominate the base intervals of $x_i$, the base intervals are dominated. Moreover, the positive (resp. negative) bridge intervals of $x_i$ and the $U$ (resp. $\overline{U}$) intervals of $x_i$ both dominate the $T$ (resp. $\overline{T}$) intervals of $x_i$. Thus, the $T$ and $\overline{T}$ intervals are all dominated. Moreover, for any variable $x_i$, the $U$ and $W$ (resp. $\overline{U}$ and $\overline{W}$) intervals of $x_i$ both dominate the $V$ (resp. $\overline{V}$) intervals of $x_i$. Thus, the $V$ and $\overline{V}$ intervals are all dominated. Finally, since $A$ satisfies $F$, each clause has at least one of its literal in $A$. Thus, each $C_j$ and $C_j'$ has at least one adjacent interval $W_j^i$ or $\overline{W}_j^i$ in $D_F(A)\setminus J$ and are therefore dominated by it, as well as the junction interval.
   \end{proof}
   
Before continuing further, let us prove a few results that are of importance in our proof. Let $K:= \frac{3mn}{2} + 3\sum_{i=1}^n \ell_i + 1$. Since the number $6mn$ of leaves attached on each dead-end interval is strictly more than $K$ (as $\ell_i\leq m$), the following holds.

\begin{remark}\label{CG:rk:dead-end}
Any dominating set of size at most $K$ contains all the $(mn+2\sum_{i=1}^n \ell_i)$ dead-end intervals.
\end{remark}

Any dominating set of size $K$ contains all the dead-end vertices. And then all the pending, dead-end, bridge, $U$, $\overline{U}$, $W$ and $\overline{W}$ intervals are dominated. So we will simply have to focus on the domination of base, $T$, $\overline{T}$, $V$, $\overline{V}$ and junction intervals (i.e. the so-called important intervals).

\begin{lemma}\label{CG:lem:atleast}
If $D$ is a dominating set of $G$, then for any variable $x_i$, $D$ contains at least $\ell_i$ intervals that dominate the $V$ and $\overline{V}$ intervals of $x_i$, and at least $\frac{m}{2}$ intervals that dominate the base intervals of $x_i$. Moreover, these two sets of intervals are disjoint, and they are intervals of $x_i$.
\end{lemma}
\begin{proof}
For any variable $x_i$, each interval $V_j^i$ (resp. $\overline{V}_j^i$) can only be dominated by $U_j^i$, $V_j^i$ or $W_j^i$ (resp. $\overline{U}_j^i$, $\overline{V}_j^i$ or $\overline{W}_j^i$). Indeed $V_j^i$ spans the left extremity of $W_j^i$ and the right extremity of $U_j^i$ and since no interval starts or ends between these two points, the interval $V_j^i$ is only adjacent to $U_j^i$ and $W_j^i$. And similarly  $\overline{V}_j^i$ is only adjacent to $\overline{U}_j^i$ and $\overline{W}_j^i$. Thus, at least $\ell_i$ intervals dominate the $V$ and $\overline{V}$ intervals of $x_i$, and they are intervals of $x_i$.
Moreover, only the base, bridge, $T$ and $\overline{T}$ intervals of $x_i$ are adjacent to the base intervals. Since each bridge interval is adjacent to two base intervals, and each $T$ and $\overline{T}$ interval of $x_i$ is adjacent to one base interval of $x_i$, $D$ must contain at least $\frac{m}{2}$ of such intervals to dominate the $m$ base intervals.
\end{proof}

Remark \ref{CG:rk:dead-end} and Lemma \ref{CG:lem:atleast} imply that any dominating set $D$ of size $K$ contains $(mn+2\sum_{i=1}^n \ell_i)$ dead-end intervals, as well as $(\ell_i+\frac{m}{2})$ intervals of $x_i$ for any variable $x_i$. Since $K=\frac{3mn}{2} + 3\sum_{i=1}^n \ell_i + 1$, this leaves only one remaining token in $D$. Thus, for any variable $x_i$ but at most one, there are exactly $(\ell_i+\frac{m}{2})$ intervals of $x_i$ in $D$. 
If there exists a variable $x_k$ such that there are more  $(\ell_i+\frac{m}{2})$ intervals of $x_k$ in $D$, then there are exactly $(\ell_k+\frac{m}{2}+1)$ intervals of $x_k$ in $D$, and we call this variable the \emph{moving variable} of $D$, denoted by $mv(D)$.

For any variable $x_i$, we denote by $X_i$ the set of positive bridge variables of $x_i$ and by $\overline{X_i}$ the set of negative bridge variables of $x_i$. Similarly, we denote by $W_i$ the set of $W$ variables of $x_i$ and by $\overline{W_i}$ the set of $\overline{W}$ variables of $x_i$. Let us now give some precision about the intervals of $x_i$ that belong to $D$.

\begin{lemma}\label{CG:lem:bridge}
If $D$ is a dominating set of size $K$, then for any variable $x_i\neq mv(D)$, either $X_i\subseteq D$ and $\overline{X_i} \cap D = \emptyset$, or $\overline{X_i}\subseteq D$ and $X_i \cap D = \emptyset$.
\end{lemma}
\begin{proof}
Since $x_i\neq mv(D)$, there are exactly $\ell_i+\frac{m}{2}$ variables of $x_i$ in $D$. Thus, by Lemma \ref{CG:lem:atleast}, exactly $\frac{m}{2}$ intervals of $x_i$ in $D$ dominate the bridge intervals of $x_i$. Only the bridge, $T$ and $\overline{T}$ intervals of $x_i$ are adjacent to the base intervals. Moreover, bridge intervals are adjacent to two base intervals and $T$ or $\overline{T}$ intervals are adjacent to only one. Since there are $m$ base intervals of $x_i$, each interval of $D$ must dominate a pair of base intervals (or none of them). So these intervals of $D$ should be some bridge intervals of $x_i$. 

Note that, by cardinality, each pair of bridge intervals of $D$ must dominate pairwise disjoint base intervals.
Let us now show by induction that these bridge intervals are either all the positive bridge intervals, or all the negative bridge intervals. We study two cases: either $X_1^i\in D$, or $X_1^i\not \in D$.

Assume that $X_1^i\in D$. In $D$, $X_1^i$ dominates $B_1^i$ and $B_2^i$. Thus, since $\overline{X}_1^i$ dominates $B_1^i$ and $\overline{X}_2^i$ dominates $B_2^i$, none of  $\overline{X}_1^i,\overline{X}_2^i$ are in $D$ (since their neighborhood in the set of base intervals is not disjoint with $X_1^i$). But $B_3^i$ (resp  $B_4^i$) is only adjacent to $\overline{X}_1^i$ and $X_2^i$ (resp. $\overline{X}_2^i$ and $X_3^i$). Thus both $X_2^i,X_3^i$ are in $D$. Suppose now that for a given $j$ such that $j$ is even and $j\leq \frac{m}{2}-2$, we have $X_j^i,X_{j+1}^i\in D$. Then, since a base interval dominated by $X_j^i$ (resp. $X_{j+1}^i$) also is dominated by $\overline{X}_{j+1}^i$ (resp. $\overline{X}_{j+2}^i$), the intervals $\overline{X}_{j+1}^i, \overline{X}_{j+2}^i$ are not in $D$. But there is a base interval adjacent only to $\overline{X}_{j+1}^i$ and $X_{j+2}^i$ (resp. $\overline{X}_{j+2}^i$ and $X_{j+3}^i$ if $j\neq \frac{m}{2}-2$, or $\overline{X}_{j+2}^i$ and $X_{j+2}^i$ if $j=\frac{m}{2}-2$). Therefore, if $j+2<\frac{m}{2}$ we have $X_{j+2}^i,X_{j+3}^i\in D$, and $X_{\frac{m}{2}}^i\in D$. By induction, if $X_1^i\in D$ then each of the $\frac{m}{2}$ positive bridge intervals belong to $D$ and thus none of the negative bridge intervals do. 

Assume now that $X_1^i\not \in D$. Then, to dominate $B_1^i$ and $B_2^i$, we must have $\overline{X}_1^i, \overline{X}_2^i\in D$. Let us show that if for a given odd $j$ such that $j\leq\frac{m}{2}-3$ we have $\overline{X}_j^i,\overline{X}_{j+1}^i\in D$, then $\overline{X}_{j+2}^i,\overline{X}_{j+3}^i\in D$. Since $\overline{X}_j^i$ (resp. $\overline{X}_{j+1}^i$) dominates base intervals also dominated by $X_{j+1}^i$ (resp. $X_{j+2}^i$), we have $X_{j+1}^i, X_{j+2}^i\not \in D$. But there exists a base interval only adjacent to $X_{j+1}^i$ and $\overline{X}_{j+2}^i$ (resp. $X_{j+2}^i$ and $\overline{X}_{j+3}^i$). Thus, $\overline{X}_{j+2}^i,\overline{X}_{j+3}^i\in D$. By induction, if $X_1^i\not \in D$ then each of the $\frac{m}{2}$ negative bridge intervals belong to $D$. Thus, none of the positive bridge intervals belong to $D$.
\end{proof}

\begin{lemma}\label{CG:lem:W}
If $D$ is a dominating set of size $K$, then for any variable $x_i\neq mv(D)$, if $X_i\subseteq D$ then $\overline{W_i}\cap D= \emptyset$, otherwise $W_i \cap  D= \emptyset$.
\end{lemma}
\begin{proof}
By Lemma~\ref{CG:lem:bridge}, $D$ either contains $X_i$ or contains $\overline{X_i}$. 

If $X_i\subset D$, Lemma \ref{CG:lem:bridge} ensures that $\overline{X}_i \cap D = \emptyset$. So the intervals $\overline{T}_{j}^i$ have to be dominated by other intervals.

By Lemma~\ref{CG:lem:atleast}, $\ell_i$ intervals must dominate the $V$ and $\overline{V}$ intervals of $x_i$. Since no interval dominates two of them, each $\overline{T}_{j}^i$ has to be dominated by an interval that is also dominating a $V$ or $\overline{V}$ interval. The only interval that dominates both $\overline{T}_{j}^i$ and a $V$ or $\overline{V}$ interval is $\overline{U}_{j}^i$. So all the $\overline{U}$ intervals are in $D$ and $\overline{W} \cap D = \emptyset$ (since the only $V$ or $\overline{V}$ interval dominated by a $\overline{W}$ interval is a $\overline{V}$ interval, which is already dominated).

Similarly if $\overline{X_i}\subset D$, Lemma \ref{CG:lem:bridge} ensures that $X_i \cap D = \emptyset$. So the intervals $T_{j}^i$ have to be dominated by other intervals. And one can prove similarly that these intervals should be the $U$ intervals and then the $W$ intervals are not in $D$.

\end{proof}

\subsection{Safeness of the reduction.}
Let $(F,A_s,A_t)$ be an instance of \satrnospace, and let $D_s=D_F(A_s)$ and $D_t=D_F(A_t)$. By Lemma \ref{CG:lem:DS}, $(G_F,D_s,D_t)$ is an instance of \dsrnospace. We can now show the first direction of our reduction.

\begin{lemma}\label{CG:lem:Dir1}
If $(F,A_s,A_t)$ is a yes-instance of \satrnospace, then $(G_F,D_s,D_t)$ is a yes-instance of \dsrnospace.
\end{lemma}
\begin{proof}
Let $(F,A_s,A_t)$ be a yes-instance of \satrnospace, and let $S=<A_1:=A_s,\ldots,A_{\ell}:=A_t>$ be the reconfiguration sequence from $A_s$ to $A_t$. We construct a reconfiguration sequence $S'$ from $D_s$ to $D_t$ by replacing any flip of variable $x_i\leadsto \overline{x_i}$ of $S$ from $A_r$ to $A_{r+1}$ by the following sequence of token slides from $D_F(A_r)$ to $D_F(A_{r+1})$ \footnote{A $\overline{x_i}\leadsto x_i$ consists in applying the converse of this sequence.}.

\begin{itemize}
\item  We perform a sequence of slides that moves the token on $J$ to $\overline{X}_1^i$. By Lemma \ref{CG:lem:connected}, $G_F$ is connected, and by Lemma \ref{CG:lem:DS}, $D_F(A_r)\setminus J$ is a dominating set. So any sequence of moves along a path from $J$ to $\overline{X}_1^i$ keeps a dominating set. 

\item For any $j$ such that $x_i\in C_j$, we first move the token from $W_j^i$ to $V_j^i$ then from $V_j^i$ to $U_j^i$. Let us show that this keeps $G_F$ dominated. The important intervals that can be dominated by $W_j^i$ are $V_j^i$, $C_j$, and $J$. The vertex $V_j^i$ is dominated anyway during the sequence since it is also dominated by $V_j^i$ and $U_j^i$. Moreover, since $x_i\leadsto \overline{x_i}$ keeps $F$ satisfied, each clause containing $x_i$ has a literal different from $x_i$ that also satisfies the clause. Thus, for each $C_j$ such that $x_i\in C_j$, there exists an interval $W_j^{i'}$ or $\overline{W}_j^{i'}$, with $i'\neq i$, that belongs to $D_F(A_r)$, and then dominates both $C_j$ and $J$ during these two moves.

\item For $j$ from $1$ to $\frac{m}{2}-1$, we apply the move $X_j^i\leadsto \overline{X}_{j+1}^i$. This move is possible since $X_j^i$ and $\overline{X}_{j+1}^i$ are neighbors in $G_F$. Let us show that this move keeps a dominating set. For $j=1$, the important intervals that are dominated by $X_1^i$ are $B_1^i$, $B_2^i$, and $T_1^i$. Since $U_1^i$ is in the current dominating set (by the second point), $T_1^i$ is dominated. Moreover $B_1^i$ is dominated by $\overline{X}_1^i$, and $B_2^i$ is a neighbor of $\overline{X}_{2}^i$. Thus, $X_1^i\leadsto \overline{X}_{2}^i$ maintains a dominating set. For $2\leq j\leq \frac{m}{2}-1$, the important intervals that are dominated by $X_j^i$ are $B_k^i$, $B_{k-2}^i$ and $T_j^i$ where $k=2j+1$ if $j$ is even and $k=2j$ otherwise. Again $T_j^i$ is dominated by the $U$ intervals. Moreover $B_{k-2}^i$ is dominated by $\overline{X}_{j-1}^i$ (on which there is a token since we perform this sequence for increasing $j$), and $B_k^i$ is also dominated by $\overline{X}_{j+1}^i$.

\item For any $j$ such that $\overline{x_i}\in C_j$, we move the token from $\overline{U}_j^i$ to $\overline{V}_j^i$ and then from $\overline{V}_j^i$ to $\overline{W}_j^i$. The important intervals  dominated by $\overline{U}_j^i$ are the intervals $\overline{T}_j^i$, $\overline{V}_j^i$. But $\overline{T}_j^i$ is dominated by a negative bridge interval, and $\overline{V}_j^i$ stays dominated by $\overline{V}_j^i$ then $\overline{W}_j^i$.

\item The previous moves lead to the dominating set $(D_F(A_{r+1})\setminus J) \cup X_{\frac{m}{2}}^i$. We finally perform a sequence of moves that slide the token on $X_{\frac{m}{2}}^i$ to $J$. It can be done since Lemma \ref{CG:lem:connected} ensures that $G_F$ is connected. And all along the transformation, we keep a dominating set by Lemma \ref{CG:lem:DS}. As wanted, it leads to the dominating set $D_F(A_{r+1})$.
\end{itemize}
\end{proof}

We now prove the other direction of the reduction. Let us prove the following lemma.

\begin{lemma}\label{CG:lem:S}
If there exists a reconfiguration sequence $S$ from $D_s$ to $D_t$, then there exists a reconfiguration sequence $S'$ from $D_s$ to $D_t$ such that for any two adjacent dominating sets $D_r$ and $D_{r+1}$ of $S'$, if both $D_r$ and $D_{r+1}$ have a moving variable, then it is the same one.
\end{lemma}
\begin{proof}
Assume that, in $S$, there exist two adjacent dominating sets $D_r$ and $D_{r+1}$ such that both $D_r$ and $D_{r+1}$ have a moving variable, and $mv(D_{r})\neq mv(D_{r+1})$. Let us modify slightly the sequence in order to avoid this move.

Since $D_r$ and $D_{r+1}$ are adjacent in $S$, we have $D_{r+1}=D_r\cup v\setminus u$, where $uv$ is an edge of $G_F$. Since $mv(D_{r})\neq mv(D_{r+1})$, $u$ is an interval of $mv(D_{r})$, and $v$ an interval of $mv(D_{r+1})$. By construction, the only edges of $G_F$ between intervals of different variables are between their $\{W, \overline{W} \}$ intervals. Thus, both $u$ and $v$ are $W$ or $\overline{W}$ intervals and, in particular they are adjacent to the junction interval $J$. 
Moreover, the only important intervals that are adjacent to $u$ (resp. $v$) are the $V$ or $\overline{V}$ intervals of the same variable as $u$, $W$ or $\overline{W}$ intervals, clause intervals, or the junction interval $J$. Since $u$ and $v$ are adjacent, and since they are both $W$ or $\overline{W}$ intervals, they cannot be adjacent to the same clause interval. But the only intervals that are potentially not dominated by $D_r\setminus u=D_{r+1}\setminus v$ should be dominated both by $u$ in $D_r$ and by $v$ in $D_{r+1}$. So these intervals are included in the set of $W$ or $\overline{W}$ intervals and the junction interval, which are all dominated by $J$. Thus, $D_r\cup J\setminus u$ is a dominating set of $G_F$. Therefore, we can add in $S$ the dominating set $D_r\cup J\setminus u$ between $D_r$ and $D_{r+1}$. This intermediate dominating set has no moving variable. By repeating this procedure while there are adjacent dominating sets in $S$ with different moving variables, we obtain the desired reconfiguration sequence $S'$.
\end{proof}

\begin{lemma}\label{CG:lem:Dir2}
If $(G_F,D_s,D_t)$ is a yes-instance of \dsrnospace, then $(F,A_s,A_t)$ is a yes-instance of \satrnospace.
\end{lemma}
\begin{proof}

Let $(G_F,D_s,D_t)$ be a yes-instance of \dsrnospace. There exists a reconfiguration sequence $S'$ from $D_s$ to $D_t$. Moreover, by Lemma \ref{CG:lem:S}, we can assume that for any two adjacent dominating sets $D_r$ and $D_{r+1}$ of $S'$, if both $D_r$ and $D_{r+1}$ have a moving variable, then it is the same one.

Let us construct a reconfiguration sequence $S$ from $A_s$ to $A_t$. To any dominating set $D$ of $G_F$, we associate a variable assignment $A(D)$ of $F$ defined as follows. For any variable $x_i\neq mv(D)$, either $X_i\subset D$ or $\overline{X}_i \subset D$ by Lemma~\ref{CG:lem:bridge}. If $X_i\subset D$ then we set $x_i=1$. Otherwise, we set $x_i=0$. 
Let $x_k$ be such that $mv(D)=x_k$ if it exists. If there exists a clause interval $C_j$ such that $W_j^k\in D$, and if for any $x_i\neq x_k$ with $x_i\in c_j$, we have $\overline{X_i}\subset D$, and for any $x_i\neq x_k$ with $\overline{x_i}\in c_j$, we have $X_i\subset D$, then we set $x_k=1$. Otherwise $x_k=0$.

Let $S$ be the sequence of assignments obtained by replacing in $S'$ any dominating set $D$ by the assignment $A(D)$. In order to conclude, we must show that the assignments associated to $D_s$ and $D_t$ are precisely $A_s$ and $A_t$. Moreover, for every dominating set $D$, the assignment associated to $D$ has to satisfy $F$. Finally, for every move in $G_F$, we must be able to associate a (possibly empty) variable flip. Let us first show a useful claim, then proceed with the end of the proof.

\begin{claim}\label{CG:claimr:same}
For any consecutive dominating sets $D_r$ and $D_{r+1}$ and any variable $x_i$ that is not the moving variable of $D_r$ nor $D_{r+1}$, the value of $x_i$ is identical in $A(D_r)$ and $A(D_{r+1})$.
\end{claim}
\begin{proofclaim}
Lemma \ref{CG:lem:bridge} ensures that for any $x_i$ such that $x_i\neq mv(D_r)$ and $x_i\neq mv(D_{r+1})$, either $X_i\subset D_r$ and $\overline{X_i}\cap D_r = \emptyset$ or $\overline{X_i}\subset D_r$ and $X_i\cap D_r = \emptyset$, and the same holds in $D_{r+1}$. Since the number of positive and negative bridge intervals is at least $2$ (since by assumption $m$ is a multiple of $4$), and $D_{r+1}$ is reachable from $D_r$ in a single step, either both $D_r$ and $D_{r+1}$ contain $X_i$, or both contain $\overline{X_i}$. Thus, by definition of $A(D)$, for any variable $x_i$ such that $x_i\neq mv(D_r)$ and $x_i\neq mv(D_{r+1})$, $x_i$ has the same value in $A(D_r)$ and $A(D_{r+1})$.
\end{proofclaim}

\begin{claim}\label{CG:claimr:AsAt}
We have $A(D_s)=A_s$ and $A(D_t)=A_t$.
\end{claim}
\begin{proofclaim}
By definition, $D_s=D_F(A_s)$ and thus $D_s$ contains the junction interval, which means that it does not have any moving variable. Moreover, $D_s$ contains $X_i$ for any variable $x_i$ such that $x_i=1$ in $A_s$ and $\overline{X_i}$ for any variable $x_i$ such that $x_i=0$ in $A_s$. Therefore, for any variable $x_i$, $x_i=1$ in $A_s$ if and only if $x_i=1$ in $A(D_s)$. Similarly, $A(D_t)=A_t$.
\end{proofclaim}

\begin{claim}\label{CG:claimr:AsatisfiesF}
For any dominating set $D$ of $S'$, $A(D)$ satisfies $F$.
\end{claim}
\begin{proofclaim}
Since the clause intervals are only adjacent to $W$ and $\overline{W}$ intervals, they are dominated by them, or by themselves in $D$. But only one clause interval can belong to $D$. Thus, for any clause interval $C_j$, if $C_j\in D$, then $C_j'$ must be dominated by a $W$ or a $\overline{W}$ interval, that also dominates $C_j$. So in any case, $C_j$ is dominated by a $W$ or a $\overline{W}$ interval. We study four possible cases and show that in each case, $c_j$ is satisfied by $A(D)$.

If $C_j$ is dominated in $D$ by an interval $W_j^i$, where $x_i\neq mv(D)$, then by Lemmas~\ref{CG:lem:bridge} and~\ref{CG:lem:W}, $X_i\subset D$ and by definition of $A(D)$, $x_i=1$. Since $W_j^i$ exists, it means that $x_i\in c_j$, thus $c_j$ is satisfied by $A(D)$. 

Similarly, if $C_j$ is dominated in $D$ by an interval $\overline{W}_j^i$, where $x_i\neq mv(D)$, then by Lemmas~\ref{CG:lem:bridge} and~\ref{CG:lem:W}, $\overline{X_i}\subset D$. So $x_i=0$. Since $\overline{W}_j^i$ exists, $\overline{x_i}\in c_j$, and therefore $c_j$ is satisfied by $A(D)$.

If $C_j$ is only dominated by $W_j^k$ in $D$, where $x_k=mv(D)$. Then, if there exists $x_i\neq x_k$ with $x_i\in c_j$ and $X_i\subset D$ (resp. $\overline{x_i}\in c_j$ and $\overline{X_i}\subset D$), then $x_i=1$ (resp. $x_i=0$) and $c_j$ is satisfied by $A(D)$. So we can assume that, for any $x_i\neq x_k$ with $x_i\in c_j$ we have $X_i\not \subset D$. By Lemma~\ref{CG:lem:bridge}, $\overline{X_i}\subset D$. And for any $x_i\neq x_k$ such that $\overline{x_i}\in c_j$ we have $\overline{X_i}\not \subset D$, and thus $X_i\subset D$. So, by definition of $A(D)$, we have $x_k=1$. Since $x_k\in c_j$ (since $W_j^k$ exists), $c_j$ is satisfied by $A(D)$.

Finally, assume that $C_j$ is only dominated by $\overline{W}_j^k$ in $D$, where $x_k=mv(D)$. If there exists $x_i\neq x_k$ such that $x_i\in c_j$ and $X_i\subset D$ (resp. $\overline{x_i}\in c_j$ and $\overline{X_i}\subset D$), then $x_i=1$ (respectively $x_i=0$) so $c_j$ is satisfied by $A(D)$. Thus, by Lemma \ref{CG:lem:bridge}, we can assume that for any $x_i\neq x_k$ such that $x_i\in c_j$ (resp. $\overline{x_i}\in c_j$), we have $\overline{X_i}\subset D$ (resp. $X_i\subset D$). Let us show that there is no clause interval $C_{j'}$ dominated by a $W_i^k$ interval of $x_k$ in $D$ and that satisfies, for any $x_i\neq x_k$, if $x_i\in c_{j'}$ then $\overline{X_i}\subset D$, and if $\overline{x_i}\in c_{j'}$ then $X_i\subset D$. This will imply $x_k=0$ by construction and then the fact that $c_j$ is satisfied.

Since $D_s$ has no moving variable, there exists a dominating set before $D$ in $S'$ with no moving variable. Let $D_r$ be the the latest in $S'$ amongst such dominating sets. By assumption, $mv(D_q)=x_k$ for any set $D_q$ that comes earlier than $D$ but later than $D_r$. Thus, by Claim \ref{CG:claimr:same}, for any variable $x_i\neq x_k$, $x_i$ has the same value in $A(D_r)$ and $A(D)$.

Now, by assumption, for any $x_i\neq x_k$ with $x_i\in c_j$ (resp. $\overline{x_i}\in c_j$) we have $\overline{X_i}\subset D$ (resp. $X_i\subset D$). Thus, since $x_i$ has the same value in $D$ and $D_r$, if $x_i\in c_j$ (resp. $\overline{x_i}\in c_j$) then $\overline{X_i}\subset D_r$ (resp. $X_i\subset D_r$) and then, by Lemma \ref{CG:lem:W}, $W_j^i\not \in D_r$ (resp. $\overline{W}_j^i\not \in D_r$). Therefore, $C_j$ is only dominated by $\overline{W}_j^k$ in $D_r$. But since $D_r$ has no moving variable, $\overline{X_k}\subset D_r$ by Lemma \ref{CG:lem:bridge} and Lemma \ref{CG:lem:W}. Thus, by Lemma \ref{CG:lem:W}, for any $j'\neq j$, $W_{j'}^k\not \in D_r$. So for any $j'\neq j$ such that $x_k\in c_{j'}$, $C_{j'}$ is dominated by at least one interval $W_{j'}^i$ or $\overline{W}_{j'}^i$ in $D_r$, where $x_i\neq x_k$. Lemma \ref{CG:lem:W} ensures that if $C_{j'}$ is dominated by $W_{j'}^i$ (resp. $\overline{W}_{j'}^i$) in $D_r$ then $X_i\subset D_r$ (resp. $\overline{X_i}\subset D_r$), and since $x_i$ has the same value in $D$ and $D_r$, it gives $X_i\subset D$ (resp. $\overline{X_i}\subset D$). Therefore, by Lemma \ref{CG:lem:bridge}, if a clause interval $C_{j'}$ is dominated by a $W$ interval of $x_k$ in $D$, then either there exists $x_i\neq x_k$ such that $x_i\in c_{j'}$ and $D(\overline{x_i})\not \subset D$, or there exists $x_i\neq x_k$ such that $\overline{x_i}\in c_j'$ and $D(x_i)\not \subset D$. By definition of $A(D)$, this implies that $x_k=0$ in $A(D)$. Since $\overline{W}_j^k$ exists, $\overline{x_k}\in c_j$ thus $c_j$ is satisfied by $A(D)$.


Therefore, every clause of $F$ is satisfied by $A(D)$, which concludes the proof.
\end{proofclaim}

\begin{claim}\label{CG:claimr:IsASwitch}
For any two dominating sets $D_r$ and $D_{r+1}$ of $S'$, either $A(D_{r+1})=A(D_{r})$, or $A(D_{r+1})$ is reachable from $A(D_r)$ with a variable flip move.\end{claim}
\begin{proofclaim}
By Claim \ref{CG:claimr:same}, for any variable $x_i$ such that $x_i\neq mv(D_r)$ and $x_i\neq mv(D_{r+1})$, $x_i$ has the same value in $A(D_r)$ and $A(D_{r+1})$.
Moreover, by definition of $S'$, if both $D_r$ and $D_{r+1}$ have a moving variable then $mv(D_r)=mv(D_{r+1})$. Therefore, at most one variable change its value between $A(D_r)$ and $A(D_{r+1})$, which concludes the proof.
\end{proofclaim}

   \end{proof}

We now have all the ingredients to prove our main result:

\begin{proof}[of Theorem~\ref{CG:thm:CG}]
Let $D_s=D_F(A_s)$ and $D_t=D_F(A_t)$. Lemma \ref{CG:lem:Dir1} and \ref{CG:lem:Dir2} ensure that $(G_F,D_s,D_t)$ is a yes-instance of \dsrnospace if and only if $(F,A_s,A_t)$ is a yes-instance of \satrnospace. Since \satr is \psc \cite{gopalan2009connectivity}, it gives the result.
   \end{proof}




\bibliographystyle{abbrv}
\bibliography{biblio.bib}

\end{document}